\documentclass[11pt,reqno]{amsart}
\usepackage{tikz}
\usepackage{subfig}
\usepackage{epstopdf}
\usepackage{epsfig}
\usepackage{math}
\graphicspath{{./figures/}}
\usepackage{tikz}
\usetikzlibrary{shapes,arrows}
\usepackage{adjustbox}
\usepackage{rotating}
\usepackage{graphicx,color}
\newcommand{\ts}{\mathsf{T}}

\tikzstyle{decision} = [diamond, draw, fill=blue!20, 
    text width=4.5em, text badly centered, node distance=3cm, inner sep=0pt]
\tikzstyle{block} = [rectangle, draw, fill=blue!20, 
    text width=5em, text centered, rounded corners, minimum height=4em]
\tikzstyle{line} = [draw, -latex']
\tikzstyle{cloud} = [draw, ellipse,fill=red!20, node distance=3cm,
    minimum height=2em]
    
\usetikzlibrary{positioning}
\tikzset{main node/.style={circle,fill=blue!20,draw,minimum size=1cm,inner sep=0pt}}

\newcommand{\black}{\color{black}}

\begin{document}
\title{Minimal Wasserstein Surfaces}
\author[Li]{Wuchen Li}
\email{wuchen@mailbox.sc.edu}
\address{Department of Mathematics, University of South Carolina, Columbia, SC 29208.}
\author[Georgiou]{Tryphon T.\ Georgiou} 
\email{tryphon@uci.edu}
\address{Department of Mechanical \& Aerospace Engineering, University of California, Irvine, CA 92697.}
\keywords{Optimal transport; Minimal surface problems.}
\thanks{W. Li's work is supported by AFOSR MURI FP 9550-18-1-502, AFOSR YIP award No. FA9550-23-1-0087, NSF RTG: 2038080, and NSF DMS-2245097. T.T.~Georgiou's work is supported by the AFOSR under grant FA9550-20-1-0029, and by the ARO under grant W911NF-22-1-0292.}
\maketitle
\begin{abstract}
In finite-dimensions, minimal surfaces that fill in the space delineated by closed curves and have minimal area arose naturally in classical physics in several contexts. No such concept seems readily available in infinite dimensions. The present work is motivated by the need for precisely such a concept that would allow natural coordinates for a surface with a boundary of a closed curve in the Wasserstein space of probability distributions ($\mathcal P_2$, space of distributions with finite second moments). The need for such a concept arose recently in stochastic thermodynamics, where the Wasserstein length in the space of thermodynamic states quantifies dissipation while ``area'' integrals (albeit, presented in a special finite-dimensional parameter setting) relate to useful work being extracted. Our goal in this work is to introduce the concept of a {\em minimal surface in} $\mathcal P_2$, and explore options for a suitable construction. To this end, we introduce the basic mathematical problem and develop two alternative formulations. Specifically,
 we cast the problem as a two-parameter Benamou-Breiner type minimization problem and derive a {\em minimal surface equation} in the form of a set of two-parameter partial differential equations. Several explicit solutions for minimal surface equations in Wasserstein spaces  are presented. These are cast in terms of covariance matrices in Gaussian distributions.  
\end{abstract}

\section{Introduction}
A familiar minimal surface is that formed by a soap film \cite{courant1940soap}, with equal pressure on either side, where the mean curvature is zero at every point. Minimal surfaces are abundant in the physical world. They originate as solutions to variational problems \cite[Chapter IV]{hilbert2021geometry}, very much like their familiar one-dimensional counterparts, curves, that minimize some type of an action integral. Indeed, for a soap film delineated by a closed curve, and similarly for more general minimal surfaces, the energy stored in the surface-tension is minimized with the area, and this is what nature ostensibly selects. Structural and mechanical advantages, improved electrical and heat transport properties, as well as noticeable aesthetics of minimal surfaces are some of the reasons that they often appear in designs and structures in several engineering disciplines including architecture and material sciences \cite{canestrino2021considerations,torquato2004minimal}.

Whereas in finite-dimensions, the geometry of minimal surfaces has long been a central theme in differential geometry, no such concept has seemingly been studied in infinite dimensional spaces.
Our interest in infinite dimensions stems from recent and fast-developing strides of the Wasserstein geometry--a natural geometry on the space of probability distributions with a plethora of applications in physics and engineering, probability theory, partial differential equations (PDEs), computer vision, machine learning, network science, control theory, filtering and data assimilation, weather forecasting and several other fields, see, e.g., \cite{evans1997partial,haker2004optimal,Villani2009_optimal,villani2021topics,chen2016relation,peyre2019computational,garbuno2020interacting,chen2021stochastic,chen2021optimal,Li1,Li2,Li3,Li4} and the references therein.

Indeed, probability distributions with finite second order moments constitute a space that can be endowed with an ``almost''  Riemannian structure, the Wasserstein space $\mathcal P_2$ \cite{villani2021topics,AGS}. This subject grew out of advances on the classical Monge-Kantorovich optimal mass transport gained by the Benamou-Brenier formula \cite{BB} and formalized in the Otto calculus \cite[Chapter 8]{villani2021topics}.
The Wasserstein distance, which quantifies distances between probability distributions, defines a {\em bona fide} metric on a suitably defined tangent space. The study of gradient and Hamiltonian flows in $\mathcal P_2$ followed, starting with \cite{JKO}, who discovered a deep connection between the entropy functional and the heat equation--the latter seen as gradient flow. 
Other useful applications relate to the modeling and computation of physical or social dynamics captured by the evolution of density functions.  
Interestingly, the far-reaching discovery by \cite{JKO} proved key in quantifying finite-time thermodynamic transitions. It pointed to dissipation being the Wasserstein length of curves in the space of probability densities \cite{chen2019stochastic} and allowed precise estimates of efficiency and power that can be delivered by thermodynamic engines; see \cite{chen2019stochastic, miangolarra2021underdamped,seifert2012stochastic} and the references therein.

Our starting point of inquiry about surfaces in $\mathcal P_2$ was the analysis of thermodynamic cycles driven by temperature gradients, in \cite{miangolarra2022geometry, miangolarra2021energy}, where Wasserstein length of the thermodynamic cycle quantifies dissipation while an ``area'' integral of an enclosed surface represents useful work being extracted. The analysis in \cite{miangolarra2022geometry,miangolarra2021energy} was based on suitably parametrized Gaussian curves and surfaces in the infinite-dimensional Wasserstein space. However, it raises the question of addressing such concepts in greater generality.
To this end, in the present work, we seek and introduce natural notions of a {\em minimal surface in} $\mathcal P_2$, and explore options for suitable constructions. That is, we explain how to define minimal surfaces bounded by closed curves in $\mathcal P_2$, derive minimal surface equations, and specialize to the finite-dimensional case of Gaussian distributions.

Specifically, we introduce two different formulations of the minimal surface problem in $\mathcal P_2$. The first, based on density trajectories, represents a two-parameter generalization of the Benamou-Brenier formula for defining Wasserstein geodesics. The second is based on a direct density-manifold formulation of the concept of area. We derive the corresponding minimal surface equations that can be viewed as a two-parameter extension of Wasserstein geodesics. For Gaussian distributions, we cast the minimal Wasserstein-surface problem directly in terms of respective covariance matrices. Several explicit solutions of minimal Wasserstein surfaces are presented, such as planes, Scherk’s surfaces, Catenoids, and Helicoids in terms of diagonal covariance matrices. 

The paper is organized as follows. In section \ref{sec2}, we first review the minimal surface equations in Euclidean space. We then derive the minimal surface equations in probability density space in section \ref{sec3}. We also write several finite-dimensional examples of minimal surface equations for Gaussian distributions in section \ref{sec4}; these are expressed in terms of covariance matrices with several explicit solutions presented.

\section{Minimal Wasserstein surfaces}\label{sec2}
We begin by reviewing the minimal surface problem in Euclidean space and then discuss examples in the three-dimensional domain. We then derive the minimal surface problems in the Wasserstein space. 
\subsection{Minimal surfaces in Euclidean space}
Consider a closed curve $\Gamma$ in $\mathbb{R}^n$. Denote by $\gamma\in C^2(\mathbb{R}^2;\mathbb{R}^d)$ a two-dimensional surface with $\Gamma$ as its boundary. As usual, we use symbols $\cdot$ or $(\cdot, \cdot)$ and $\|\;\;\|$ to denote the Euclidean inner product and the Euclidean norm. The minimal surface problem in $\mathbb{R}^n$ refers to identifying such a surface $\gamma$ having minimal area, i.e., it refers to the following variational problem: determine
\begin{equation}\label{E_min_surface}
\min\; \{\mbox{ Area}(\gamma) \mid \mbox{ over all surfaces }\gamma \mbox{ enclosed by }\Gamma \}, 
\end{equation}
and characterize any minimizer. 
The formula for the area can be expressed as 
\begin{align*}
\textrm{Area}(\gamma)=&\int\int \sqrt{\mathrm{det}\begin{pmatrix} 
\|\partial_s \gamma\|^2 & \partial_s\gamma\cdot \partial_t\gamma \\
 \partial_s\gamma\cdot \partial_t\gamma &  \|\partial_t\gamma\|^2
\end{pmatrix}}
ds dt\\
=&\int \int \sqrt{\|\partial_s\gamma(s,t)\|^2\|\partial_t\gamma(s,t)\|^2-(\partial_s\gamma(s,t)\cdot \partial_t\gamma(s,t))^2 }ds dt
\end{align*}
where the infimum is sought among all (suitably smooth) parametric descriptions $\gamma(s,t)$ for the two-dimensional surface. The minimizer of \eqref{E_min_surface} is shown to satisfy a suitable minimal surface equation. 

To fix ideas, we discuss an example. 
Consider a two-dimensional surface embedded in $\mathbb R^3$, described by  
\begin{equation*}
  \gamma(s,t)=(s, t, z(s,t)) \in\mathbb R^3, 
\end{equation*}
with $z=z(s,t)$ representing the surface parameterized by $(s,t)\in [0,1]^2$. Suppose $z(0, t)$, $z(1, t)$, $z(s, 0)$, $z(s,1)$ are given curves, for $s, t\in [0,1]$.  In this case, 
\begin{equation*}
\partial_s\gamma=(1, 0, \partial_s z(s,t)),\quad   \partial_t\gamma=(0, 1, \partial_t z(s,t)).  
\end{equation*}
After standard computations, the minimization surface problem \eqref{E_min_surface} reduces to 
\begin{equation*}
    \min_{z}\quad \int_0^1\int_0^1 \sqrt{1+|\partial_sz(s,t)|^2+|\partial_t z(s,t)|^2}dsdt, 
\end{equation*}
where the minimizer is among all continuous differentiable two-parameter curves $z$ with fixed boundaries $z(0, t)$, $z(1, t)$, $z(s, 0)$, $z(s,1)$, and $s, t\in [0,1]$. It easily follows that the  critical-point satisfies 
\begin{equation}\label{ms}
\partial_s(\frac{\partial_sz}{\sqrt{1+|\partial_s z|^2+|\partial_t z|^2}})+\partial_t(\frac{\partial_tz}{\sqrt{1+|\partial_s z|^2+|\partial_t z|^2}})=0. 
\end{equation}
By expanding equation \eqref{ms}, we have 
\begin{equation*}
(1+|\partial_tz|^2)\partial^2_{ss}z-2\partial_tz\partial_sz\partial^2_{st}z+ (1+|\partial_sz|^2)\partial^2_{tt}z=0.    
\end{equation*}
{The plane is a trivial solution if the boundary set consists of a linear combination of lines (on a plane), i.e.,  when $\partial^2_{ss}z=\partial^2_{tt}z=\partial^2_{ts}z=0$, which indicates that $z$ is linear w.r.t.\ $s$, $t$. In general
there may be multiple non-trivial solutions for the nonlinear PDE \eqref{ms}.}

\subsection{Minimal surfaces in density space, I}\label{minimal}
We now turn to minimal surface problems in the space of probability densities. For simplicity of discussion, we consider Wasserstein-2 geodesics that delineate a rectangular boundary of interest. We refer readers to see definitions of Wasserstein-2 geodesics in \cite{AGS}. 

Consider absolutely continuous measures with density functions $\rho(0, 0, x)$, $\rho(0,1,x)$, $\rho(1,0,x)$, $\rho(1,1,x)$, $x\in\mathbb R^n$, with finite second-order moments. Let $\rho(0, t, x)$ represent a Wasserstein-2 geodesic connecting  $\rho(0, 0, x)$ and $\rho(0,1,x)$, for $t\in [0,1]$.  Similarly, let $\rho(1, t, x)$ be the Wasserstein-2 geodesic connecting  $\rho(1, 0, x)$ and $\rho(1,1,x)$, for $t\in [0,1]$,  $\rho(s, 0, x)$ be the Wasserstein-2 geodesic connecting  $\rho(0, 0, x)$ and $\rho(1,0,x)$, for $s\in [0,1]$, and $\rho(s, 1, x)$ be the Wasserstein-2 geodesic connecting $\rho(0, 1, x)$ and $\rho(1,1,x)$, for $s\in [0,1]$. These Wasserstein geodesic curves delineate a rectangular region in the probability density space, as depicted in the following schematic: 
\begin{figure}[H]
\begin{tikzpicture}
\draw (0,0) node[below left] {$\rho(0,0,\cdot)$} --
(1,0) node[below right] {$\rho(1,0,\cdot)$} --
(1,1) node[above right] {$\rho(1,1,\cdot)$} --
(0,1) node[above left] {$\rho(0,1,\cdot)$} -- cycle;
\end{tikzpicture}
\end{figure}

Next, our goal is to describe a two-parameter surface $\rho(s,t,x)$, $(s,t)\in [0,1]^2$ in the probability density space $\mathcal P_2$ that defines in a suitable sense a minimal surface with the prescribed boundary. From now on, we use $\int$ to represent $\int_{\mathbb{R}^n}$. 
\begin{problem}[Minimal surfaces in density space]\label{prob1}{\em
Consider two-parameter families of probability densities $\rho\colon [0,1]^2\times \mathbb{R}^n\rightarrow \mathbb{R}$, and define the following minimization problem:  
\begin{subequations}\label{mws}
\begin{equation}\label{mws1}
\inf_{v_s, v_t,\rho}\quad \int_0^1\int_0^1\sqrt{\int \|v_s\|^2\rho dx\cdot \int\|v_t\|^2\rho dx-(\int v_s\cdot v_t\rho dx)^2 }ds dt,
\end{equation}
where the minimum is taken over the continuous differentiable density surface $\rho(s,t,\cdot)\in \mathcal{P}_2(\mathbb{R}^n)$, $s, t\in [0,1]$, and a choice of continuous differentiable vector fields $v_s$, 
$v_t\colon [0,1]^2\times\mathbb{R}^n\rightarrow\mathbb{R}$, such that the following two continuity equations hold 
\begin{equation}\label{mws2}
\begin{aligned}
&\partial_s\rho(s,t,x)+\nabla\cdot(\rho(s,t,x) v_s(s,t,x))=0,\\
&\partial_t\rho(s,t,x)+\nabla\cdot(\rho(s,t,x) v_t(s,t,x))=0,
\end{aligned}
\end{equation}
over the four fixed boundaries in the probability density space: 
\begin{equation*}   \textrm{$\rho(0, t, \cdot)$, $\rho(1, t, \cdot)$, $\rho(s, 0, \cdot)$, $\rho(s,1, \cdot)$, where $s, t\in [0,1]$.}
\end{equation*}
\end{subequations}
}\end{problem}

Note that in the above formulation,
\begin{equation*}
A(s,t)=\sqrt{\int \|v_s\|^2\rho dx\cdot \int\|v_t\|^2\rho dx-(\int v_s\cdot v_t\rho dx)^2 },
\end{equation*}
represents the area element in the space of probability densities, while the total area delineated by the boundaries is defined as
\begin{equation*}
\int_0^1\int_0^1 A(s,t)dsdt. 
\end{equation*}
Thus, the minimization problem \eqref{mws} seeks a two-parameter family (surface) in the probability density space, which minimizes the total area enclosed by density boundary curves. Equations \eqref{mws} can be viewed as a two-parameter generalization of the Benamou-Breiner formula \cite{BB} in deriving geodesics equations in Wasserstein-2 space. 

We next derive optimality conditions for the above variational problem in \eqref{mws}. 
\begin{proposition}
{Assume that the following equations hold
\begin{subequations}\label{csaddle}
\begin{equation}
\left\{\begin{aligned}
  &A^{-1}\Big[v_s\int {\|v_t\|^2}{\rho}dx-v_t\int {v_s\cdot v_t}{\rho}dx\Big]=\nabla\Phi_s,\\
   &A^{-1}\Big[v_t\int {\|v_s\|^2}{\rho}dx-v_s\int {v_s\cdot v_t}{\rho}dx\Big]=\nabla\Phi_t, 
\end{aligned}\right.
\end{equation}
and 
\begin{equation}
\left\{\begin{aligned}
&\partial_s\Phi_s+\partial_t\Phi_t+A^{-1}\Big[\frac{1}{2}\|v_s\|^2\int {\|v_t\|^2}{\rho}dx +\frac{1}{2}\|v_t\|^2\int{\|v_s\|^2}{\rho}dx - {v_t\cdot v_s}\int {v_t\cdot v_s}{\rho} dx\Big]=0,\\
&   \partial_s\rho+\nabla\cdot (\rho v_s)=0,\\
&       \partial_t\rho+\nabla\cdot (\rho v_t)=0,
\end{aligned}\right.\
 \end{equation}
     \end{subequations}
  for suitable functions $\Phi_s$, $\Phi_t\colon [0,1]^2\times \mathbb{R}^n\rightarrow\mathbb{R}$ (Lagrange multipliers), vector fields $v_s$, $v_t$, and a two-parameter family $\rho$ as before. Then $\rho$ is a critical-point of the variational problem \eqref{mws}.}
\end{proposition}

\begin{proof}
Denote momenta fields
\begin{equation*}
m_s(s,t,x)=\rho(s,t,x)v_s(s,t,x), \quad m_t(s,t,x)=\rho(s,t,x)v_t(s,t,x),
\end{equation*}
and write the area element in terms of $m_s,m_t$ and $\rho$, as
\begin{equation*}
A(m_s, m_t,\rho):=\sqrt{\int \frac{\|m_s\|^2}{\rho} dx\cdot \int\frac{\|m_t\|^2}{\rho} dx-(\int \frac{m_s\cdot m_t}{\rho} dx)^2 }.
\end{equation*}
Now, consider
\begin{subequations}\label{mwsn}
\begin{equation}\label{mwsn1}
\inf_{m_s, m_t,\rho}\quad \int_0^1\int_0^1A(m_s, m_t,\rho)ds dt,
\end{equation}
s.t. 
\begin{equation}\label{mwsn2}
\begin{aligned}
&\partial_s\rho+\nabla\cdot m_s=0,\quad \partial_t\rho+\nabla\cdot m_t=0,
\end{aligned}
\end{equation}
\end{subequations}
and introduce Lagrange multipliers $\Phi_s$, $\Phi_t\colon [0,1]^2\times \mathbb{R}^n\rightarrow\mathbb{R}$ for the constraints \eqref{mwsn2}, in the equations for $\partial_s\rho_s$, $\partial_t\rho_t$, respectively. The variational problem \eqref{mwsn} satisfies the following saddle point problem 
\begin{equation*}
\inf_{m_s, m_t, \rho}\sup_{\Phi_s, \Phi_t}\quad \mathcal{L}(m_s,m_t,\rho, \Phi_s, \Phi_t), 
\end{equation*}
where
\begin{equation*}
\begin{split}
 \mathcal{L}(m_s,m_t,\rho, \Phi_s, \Phi_t):=& \quad \int_0^1\int_0^1 A(m_s, m_t, \rho)ds dt\\
&+\int_0^1\int_0^1 \int\Phi_s(\partial_s\rho+\nabla\cdot m_s)+\Phi_t(\partial_t\rho+\nabla\cdot m_t) dx dsdt. 
\end{split}
 \end{equation*}
 The saddle point system satisfies 
 \begin{equation*}
\left\{\begin{aligned}
 \frac{\delta}{\delta \rho}\mathcal{L}=0,\\
  \frac{\delta}{\delta m_s}\mathcal{L}=0,\\
    \frac{\delta}{\delta m_t}\mathcal{L}=0,\\
      \frac{\delta}{\delta \Phi_s}\mathcal{L}=0,\\
        \frac{\delta}{\delta \Phi_t}\mathcal{L}=0,
\end{aligned}\right.\quad \Rightarrow \quad  
\left\{\begin{aligned}
 \frac{\delta}{\delta \rho} A(m_s,m_t, \rho)-\partial_s\Phi_s-\partial_t\Phi_t=0,\\
  \frac{\delta}{\delta m_s}A(m_s,m_t, \rho)-\nabla\Phi_s=0,\\
    \frac{\delta}{\delta m_t}A(m_s,m_t, \rho)-\nabla\Phi_t=0,\\
   \partial_s\rho+\nabla\cdot m_s=0,\\
       \partial_t\rho+\nabla\cdot m_t=0.
\end{aligned}\right.\
 \end{equation*} 
In above, $\frac{\delta}{\delta(\cdot)}(\cdot)$ denotes the $L^2$ first-variation operator. We next derive the first variation operator of $A$. 
Denote $\delta m_s$, $\delta m_t$, and $\delta \rho$ as variations of $m_s$, $m_t$, $\rho$, respectively. Denote $\epsilon\in \mathbb{R}_+$. Thus 
\begin{equation*}
\begin{split}
&\frac{d}{d\epsilon}A(m_s+\epsilon \delta m_s, m_t+\epsilon \delta m_t, \rho+\epsilon \delta \rho)|_{\epsilon=0}\\
=&\frac{d}{d\epsilon}\sqrt{\int \frac{\|m_s+\epsilon \delta m_s\|^2}{\rho+\epsilon \delta\rho} dx\cdot \int\frac{\|m_t+\epsilon \delta m_t\|^2}{\rho+\epsilon \delta\rho} dx-(\int \frac{(m_s+\epsilon \delta m_s)\cdot( m_t+\epsilon \delta m_t)}{\rho+\epsilon \delta\rho} dx)^2 }|_{\epsilon=0}\\
=& \frac{1}{2}A^{-1}\int \Big\{\quad \Big[\frac{2m_s}{\rho}\int\frac{\|m_t\|^2}{\rho}dx-\frac{2m_t}{\rho}\int\frac{m_s\cdot m_t}{\rho}dx\Big]\cdot\delta m_s\\
&\hspace{1.5cm}+\Big[\frac{2m_t}{\rho}\int\frac{\|m_s\|^2}{\rho}dx-\frac{2m_s}{\rho}\int\frac{m_s\cdot m_t}{\rho}dx\Big]\cdot\delta m_t \\
&\hspace{1.5cm}+\Big[-\frac{\|m_s\|^2}{\rho^2}\int\frac{\|m_t\|^2}{\rho}dx -\frac{\|m_t\|^2}{\rho^2}\int\frac{\|m_s\|^2}{\rho}dx +2\frac{m_t\cdot m_s}{\rho^2}\int\frac{m_t\cdot m_s}{\rho} dx\Big]\delta\rho\Big\}dx.
\end{split}
\end{equation*} 
Using the fact that $\frac{d}{d\epsilon}A(m_s+\epsilon \delta m_s, m_t+\epsilon \delta m_t, \rho+\epsilon \delta \rho)|_{\epsilon=0}=\int \frac{\delta}{\delta m_s}A \cdot \delta m_s+\frac{\delta}{\delta m_t}A\cdot\delta m_t+\frac{\delta}{\delta \rho}A\cdot \delta \rho dx$, we obtain  
\begin{equation*}
\left\{\begin{aligned}
\frac{\delta}{\delta m_s}A=&A^{-1}\Big[\frac{m_s}{\rho}\int\frac{\|m_t\|^2}{\rho}dx-\frac{m_t}{\rho}\int\frac{m_s\cdot m_t}{\rho}dx\Big],\\
\frac{\delta}{\delta m_t}A=&A^{-1}\Big[\frac{m_t}{\rho}\int\frac{\|m_s\|^2}{\rho}dx-\frac{m_s}{\rho}\int\frac{m_s\cdot m_t}{\rho}dx\Big],\\
\frac{\delta}{\delta \rho}A=&A^{-1}\Big[-\frac{\|m_s\|^2}{2\rho^2}\int\frac{\|m_t\|^2}{\rho}dx -\frac{\|m_t\|^2}{2\rho^2}\int\frac{\|m_s\|^2}{\rho}dx +\frac{m_t\cdot m_s}{\rho^2}\int\frac{m_t\cdot m_s}{\rho} dx\Big].
\end{aligned}\right.
\end{equation*}
We derive the critical-point system by using the fact again that $m_s=\rho_s v_s$, $m_t=\rho_tv_t$.  
\end{proof}

\subsection{Minimal surfaces in density space, II}
A somewhat more restrictive but potentially more natural formulation consists in restricting attention to gradient vector fields as is customary in Wasserstein manifolds.
We proceed to explain and explore this.

Our construction here seeks vector fields $v_s$, $v_t$ being gradient vector fields. That is, it postulates the existence of potential functions $\Psi_s$, $\Psi_t$, such that $v_s=\nabla\Psi_s$ and $v_t=\nabla\Psi_t$. Thus, our problem is to find potential functions $\Psi_s$, $\Psi_t$, such that the surface area enclosed by Wasserstein geodesics is minimized--the same Wasserstein rectangular curves from subsection \ref{minimal}.
\black


\begin{problem}[Minimal surfaces on Wasserstein manifold]\label{prob2}
{\em 
Consider two-parameter families of probability densities $\rho\colon [0,1]^2\times \mathbb{R}^n\rightarrow \mathbb{R}$, and define the following minimization problem:    
\begin{subequations}\label{mwss}
\begin{equation}\label{mwss1}
\inf_{\Psi_s, \Psi_t,\rho}\quad \int_0^1\int_0^1\sqrt{\int \|\nabla\Psi_s\|^2\rho dx\cdot \int\|\nabla\Psi_t\|^2\rho dx-(\int \nabla\Psi_s\cdot \nabla\Psi_t\rho dx)^2 }ds dt,
\end{equation}
where the minimum is taken among density surfaces $\rho(s,t,\cdot)\in \mathcal{P}_2(\mathbb{R}^n)$, $s, t\in [0,1]$, and corresponding vector fields $\nabla\Psi_s$, 
$\nabla\Psi_t\colon [0,1]^2\times\mathbb{R}^n\rightarrow\mathbb{R}^n$, such that two continuity equations hold 
\begin{equation}\label{mwss2}
\begin{aligned}
&\partial_s\rho(s,t,x)+\nabla\cdot(\rho(s,t,x) \nabla\Psi_s(s,t,x))=0,\\
&\partial_t\rho(s,t,x)+\nabla\cdot(\rho(s,t,x) \nabla\Psi_t(s,t,x))=0,
\end{aligned}
\end{equation}
with fixed boundaries in the probability density space: 
\begin{equation*}   \textrm{$\rho(0, t,\cdot)$, $\rho(1, t,\cdot)$, $\rho(s, 0,\cdot)$, $\rho(s,1,\cdot)$, where $s, t\in [0,1]$.}
\end{equation*}
\end{subequations}
}
\end{problem}
Similarly, as before, we derive minimal surface equations in the form of the critical-point system of equations for our variational problem \eqref{mwss}. 
\begin{proposition}{Consider functions  $\rho,\Psi,\Phi,\tilde\Psi,\tilde\Phi$, in parameters $(s,t,x)$, from $[0,1]\times[0,1]\times \mathbb{R}^n\rightarrow\mathbb{R}$,
and define}
\begin{equation*}
B(\Psi_s, \Psi_t,\rho):=\sqrt{\int {\|\nabla\Psi_s\|^2}{\rho} dx\cdot \int{\|\nabla\Psi_t\|^2}{\rho} dx-(\int {\nabla\Psi_s\cdot \nabla\Psi_t}{\rho} dx)^2 },
\end{equation*}
If the following equations hold,
\begin{equation*}
\left\{\begin{aligned}
  &B^{-1}\nabla\cdot\Big(\rho \Big[\nabla\Psi_s\int {\|\nabla\Psi_t\|^2}{\rho}dx-\nabla\Psi_t\int {\nabla\Psi_s\cdot \nabla\Psi_t}{\rho}dx\Big]\Big)=\nabla\cdot(\rho\nabla\tilde\Phi_s),\\
&B^{-1}\nabla\cdot\Big(\rho\Big[\nabla\Psi_t\int {\|\nabla\Psi_s\|^2}{\rho}dx-\nabla\Psi_s\int {\nabla\Psi_s\cdot \nabla\Psi_t}{\rho}dx\Big]\Big)=\nabla\cdot(\rho\nabla\tilde\Phi_t), 
\end{aligned}\right.
\end{equation*}
and 
\begin{equation*}
\left\{\begin{aligned}
&\partial_s\tilde\Phi_s+\partial_t\tilde\Phi_t+(\nabla\tilde\Phi_s, \nabla\Psi_s)+(\nabla\tilde\Phi_t, \nabla\Psi_t)\\
&\hspace{2cm}-B^{-1}\Big[\frac{1}{2}{\|\nabla\Psi_s\|^2}\int{\|\nabla\Psi_t\|^2}{\rho}dx +\frac{1}{2}{\|\nabla\Psi_t\|^2}\int{\|\nabla\Psi_s\|^2}{\rho}dx\\
&\hspace{3.5cm}-{\nabla\Psi_t\cdot \nabla\Psi_s}\int{\nabla\Psi_t\cdot \nabla\Psi_s}{\rho} dx\Big]=0,\\
&   \partial_s\rho+\nabla\cdot (\rho_s\nabla\Psi_s)=0,\\
&       \partial_t\rho+\nabla\cdot (\rho_t\nabla\Psi_t)=0,
\end{aligned}\right.\
 \end{equation*}
 {then $\rho,\Psi,\Phi$ solve Problem \ref{prob2}, which is a critical-point system of variational problem \eqref{mwss}.}
\end{proposition}
\begin{proof}
As before, consider
\begin{subequations}\label{mwssn}
\begin{equation}\label{mwssn1}
\inf_{\Psi_s, \Psi_t,\rho}\quad \int_0^1\int_0^1B(\Psi_s, \Psi_t,\rho)ds dt,
\end{equation}
s.t. 
\begin{equation}\label{mwssn2}
\begin{aligned}
&\partial_s\rho+\nabla\cdot (\rho\nabla\Psi_s)=0,\quad \partial_t\rho+\nabla\cdot(\rho\nabla\Psi_t)=0. 
\end{aligned}
\end{equation}
\end{subequations}
Introduce Lagrange multipliers $\tilde\Phi_s$, $\tilde\Phi_t\colon [0,1]^2\times \mathbb{R}^n\rightarrow\mathbb{R}$ for the constraints \eqref{mwssn2}. 
Thus, the variational problem \eqref{mwssn} leads to the following saddle point problem 
\begin{equation*}
\inf_{\Psi_s, \Psi_t, \rho}\sup_{\tilde\Phi_s, \tilde\Phi_t}\quad \mathcal{L}_1(\Psi_s,\Psi_t,\rho, \tilde\Phi_s, \tilde\Phi_t), 
\end{equation*}
where
\begin{equation*}
\begin{split}
 \mathcal{L}_1(\Psi_s,\Psi_t,\rho, \tilde\Phi_s, \tilde\Phi_t):=& \quad \int_0^1\int_0^1 B(\Psi_s, \Psi_t, \rho)ds dt\\
&+\int_0^1\int_0^1 \int\tilde\Phi_s(\partial_s\rho+\nabla\cdot (\rho\nabla\Psi_s))+\tilde\Phi_t(\partial_t\rho+\nabla\cdot (\rho\nabla\Psi_t)) dx dsdt. 
\end{split}
 \end{equation*}
 The saddle point system is 
 \begin{equation*}
\left\{\begin{aligned}
 \frac{\delta}{\delta \rho}\mathcal{L}_1=0,\\
  \frac{\delta}{\delta \Psi_s}\mathcal{L}_1=0,\\
    \frac{\delta}{\delta \Psi_t}\mathcal{L}_1=0,\\
      \frac{\delta}{\delta \tilde\Phi_s}\mathcal{L}_1=0,\\
        \frac{\delta}{\delta \tilde\Phi_t}\mathcal{L}_1=0,
\end{aligned}\right.\quad \Rightarrow \quad  
\left\{\begin{aligned}
 \frac{\delta}{\delta \rho} B(\Psi_s,\Psi_t, \rho)-(\nabla\tilde\Phi_s,\nabla\Psi_s)-(\nabla\tilde\Phi_t, \nabla\Psi_t)-\partial_s\tilde\Phi_s-\partial_t\tilde\Phi_t=0,\\
  \frac{\delta}{\delta \Psi_s}B(\Psi_s,\Psi_t, \rho)+\nabla\cdot(\rho \nabla\tilde\Phi_s)=0,\\
    \frac{\delta}{\delta \Psi_t}B(\Psi_s,\Psi_t, \rho)+\nabla\cdot(\rho \nabla\tilde\Phi_t)=0,\\
   \partial_s\rho+\nabla\cdot (\rho \nabla\Psi_s)=0,\\
       \partial_t\rho+\nabla\cdot (\rho\nabla\Psi_t)=0.
\end{aligned}\right.\
 \end{equation*} 
We next derive the first variation operator of $B$. 
Denote $\delta \Psi_s$, $\delta \Psi_t$, and $\delta \rho$ as variations of $\Psi_s$, $\Psi_t$, $\rho$, respectively. Denote $\epsilon\in \mathbb{R}_+$. Thus,
\begin{equation*}
\begin{split}
&\frac{d}{d\epsilon}B(\Psi_s+\epsilon \delta \Psi_s, \Psi_t+\epsilon \delta \Psi_t, \rho+\epsilon \delta \rho)|_{\epsilon=0}\\
=&\frac{d}{d\epsilon}\Big(\int {\|\nabla\Psi_s+\epsilon \nabla\delta\Psi_s\|^2}{(\rho+\epsilon \delta\rho)} dx\cdot \int{\|\nabla\Psi_t+\epsilon \delta \Psi_t\|^2}{(\rho+\epsilon \delta\rho)} dx\\
&\hspace{2cm}-(\int {(\nabla\Psi_s+\epsilon \nabla\delta \Psi_s)\cdot( \nabla\Psi_t+\epsilon \nabla \delta \Psi_t)}{(\rho+\epsilon \delta\rho}) dx)^2 \Big)^{\frac{1}{2}}|_{\epsilon=0}\\
=& \frac{1}{2}B(\Psi_s,\Psi_t,\rho)^{-1}\int \Big\{\quad \Big[2\nabla\Psi_s\int {\|\nabla\Psi_t\|^2}{\rho}dx-{2\nabla\Psi_t}\int{\nabla\Psi_s\cdot \nabla\Psi_t}{\rho}dx\Big]\cdot \nabla\delta \Psi_s \rho\\
&\hspace{3.5cm}+\Big[{2\nabla\Psi_t}\int{\|\nabla\Psi_s\|^2}{\rho}dx-{2\nabla\Psi_s}\int{\nabla\Psi_s\cdot \nabla\Psi_t}{\rho}dx\Big]\cdot \nabla\delta\Psi_t \rho \\
&\hspace{3.5cm}+\Big[{\|\nabla\Psi_s\|^2}\int{\|\nabla\Psi_t\|^2}{\rho}dx +{\|\nabla\Psi_t\|^2}\int{\|\nabla\Psi_s\|^2}{\rho}dx \\
&\hspace{4.2cm}-2{\nabla\Psi_t\cdot \nabla\Psi_s}\int{\nabla\Psi_t\cdot \nabla\Psi_s}{\rho} dx\Big]\delta\rho\Big\}dx.
\end{split}
\end{equation*} 
Using the fact that $\frac{d}{d\epsilon}B(\Psi_s+\epsilon \delta \Psi_s, \Psi_t+\epsilon \delta \Psi_t, \rho+\epsilon \delta \rho)|_{\epsilon=0}=\int \frac{\delta}{\delta \Psi_s}B \cdot \delta \Psi_s+\frac{\delta}{\delta \Psi_t}B\cdot\delta \Psi_t+\frac{\delta}{\delta \rho}B\cdot \delta \rho dx$, we derive  
\begin{equation*}
\left\{\begin{aligned}
\frac{\delta}{\delta \Psi_s}B=&-B^{-1}\nabla\cdot\Big(\rho \Big[{\nabla\Psi_s}\int{\|\nabla\Psi_t\|^2}{\rho}dx-{\nabla\Psi_t}\int{\nabla\Psi_s\cdot \nabla\Psi_t}{\rho}dx\Big]\Big),\\
\frac{\delta}{\delta \Psi_t}B=&-B^{-1}\nabla\cdot\Big(\rho \Big[{\nabla\Psi_t}\int{\|\nabla\Psi_s\|^2}{\rho}dx-{\nabla\Psi_s}\int{\nabla\Psi_s\cdot \nabla\Psi_t}{\rho}dx\Big]\Big),\\
\frac{\delta}{\delta \rho}B=&B^{-1}\Big[\frac{1}{2}{\|\nabla\Psi_s\|^2}\int{\|\nabla\Psi_t\|^2}{\rho}dx+\frac{1}{2}{\|\nabla\Psi_t\|^2}\int{\|\nabla\Psi_s\|^2}{\rho}dx\\
&\hspace{1cm}-{\nabla\Psi_t\cdot \nabla\Psi_s}\int{\nabla\Psi_t\cdot \nabla\Psi_s}{\rho} dx\Big].
\end{aligned}\right.
\end{equation*}
Thus we derive the  critical-point system.  
\end{proof}

\section{Minimal surface equations in Lagrangian coordinates}\label{sec3}
In this section we present examples of minimal Wasserstein surface equations. We first discuss the simplest case of a one-dimensional sample space. In this case the minimal Wasserstein surface problems, Problems \ref{prob1} and \ref{prob2} (equations \eqref{mws} and \eqref{mwss}), coincide. We next discuss generalized minimal Wasserstein surface equations in high dimensional sample space. 

\subsection{One-dimensional sample space}
\black
{Define a monotone in $x$
mapping $T(s,t,x)$, with $T\colon [0,1]\times [0,1]\times \mathbb{R}^1\rightarrow\mathbb{R}^1$,} that is a mapping such that $\rho(s,t,\cdot)= T(s,t,\cdot)_\# \rho(0,0,\cdot)$, equivalently, 
\begin{equation*}
\rho(s,t, T(s,t,x)) \partial_x T(s,t,x)=\rho(0,0,x). 
\end{equation*}
{In this case, the ``two-parameter mapping-function'' is explicitly computable. Specifically,}
\begin{equation*}
\frac{d}{dx}F_{\rho_{st}}(T(s,t,x))=\rho(0,0,x)=\rho_{00}(x)\quad \Rightarrow \quad F_{\rho_{st}}(T(s,t,x))=F_{\rho_{00}}(x), 
\end{equation*}
where $F_{\rho_{st}}(x)=\int_{-\infty}^x\rho(s,t,y)dy$ and $F_{\rho_{00}}(x)=\int_{-\infty}^x\rho(0,0,y)dy$ are cumulative distributions of $\rho_{st}$, $\rho_{00}$, respectively, and $F_{\rho_{st}}^{-1}$, $F_{\rho_{00}}^{-1}$ are their inverse functions. In other words, 
\begin{equation*}
T(s,t,x)=F_{\rho_{st}}^{-1}(F_{\rho_{00}}(x)). 
\end{equation*}
We now reformulate the minimal surface problem \eqref{mws} or \eqref{mwss} in terms of inverse cumulative functions. 
\begin{proposition}[One-dimensional minimal Wasserstein surfaces]
Variational problem \eqref{mws} or \eqref{mwss} can be formulated as follows:  
\begin{equation}\label{lmws1}
\inf_{F_{\rho_{st}}^{-1}}\quad \int_0^1\int_0^1\sqrt{\int \|\partial_s F^{-1}_{\rho_{st}}(x)\|^2dx\cdot \int\|\partial_t F^{-1}_{\rho_{st}}(x)\|^2 dx-(\int \partial_s F^{-1}_{\rho_{st}}(x)\cdot \partial_t F^{-1}_{\rho_{st}}(x) dx)^2 }ds dt,
\end{equation}
where the minimum is taken amongst all inverse cumulative distributions $F^{-1}_{\rho_{st}}$, $s, t\in [0,1]$, with fixed boundaries: 
\begin{equation*}   \textrm{$F^{-1}_{\rho_{0t}}(\cdot)$, $F^{-1}_{\rho_{1t}}(\cdot)$, $F^{-1}_{\rho_{s0}}(\cdot)$, $F^{-1}_{\rho_{s1}}(\cdot)$, where $s, t\in [0,1]$.}
\end{equation*}
\end{proposition}
\begin{proof}
The continuity equation \eqref{mws2} can be reformulated as follows. 
Denote
\begin{equation*}
\frac{\partial}{\partial s}T(s,t,x)=v_s(s,t, T(s,t,x)),\quad  \frac{\partial}{\partial t}T(s,t,x)=v_t(s,t, T(s,t,x)).
\end{equation*}
And the objective functional \eqref{mws1} can be reformulated in terms of inverse cumulative functions,
\begin{equation*}
\begin{split}
\int \|v_s\|^2\rho dx=& \int \|v_s(s,t,y)\|^2\rho(s,t,y)dy\\
=& \int \|v_s(s,t,T(s,t,x))\|^2\rho(s,t,T(s,t,x))dT(s,t,x)\\
=& \int \|v_s(s,t,T(s,t,x))\|^2\rho(s,t,T(s,t,x))\partial_xT(s,t,x)dx\\
=&\int \|\partial_sT(s,t,x)\|^2\rho_{00}(x)dx\\
=&\int \|\partial_s F^{-1}_{\rho_{st}}(F_{\rho_{00}}(x))\|^2dF_{\rho_{00}}(x)\\
=&\int \|\partial_s F^{-1}_{\rho_{st}}(z)\|^2dz, 
\end{split}
\end{equation*}
where we let $y=T(s,t,x)$ in the first equality and $z=F_{\rho_{00}}(x)$ in the last equality. 
Similarly, 
\begin{equation*}
\int \|v_t\|^2\rho dx=\int \|\partial_t F^{-1}_{\rho_{st}}(z)\|^2dz, 
\end{equation*}
and 
\begin{equation*}
\begin{split}
\int v_s\cdot v_t\rho dx=& \int  v_s(s,t,y)\cdot v_t(s,t,y)\rho(s,t,y)dy\\
=& \int v_s(s,t,T(s,t,x))\cdot v_t(s,t, T(s,t,x))\rho(s,t,T(s,t,x))dT(s,t,x)\\
=& \int v_s(s,t,T(s,t,x))\cdot v_t(s,t, T(s,t,x))\rho(s,t,T(s,t,x))\partial_xT(s,t,x)dx\\
=&\int \partial_sT(s,t,x)\cdot \partial_tT(s,t,x)\rho_{00}(x)dx\\
=&\int \partial_s F^{-1}_{\rho_{st}}(F_{\rho_{00}}(x))\cdot \partial_t F^{-1}_{\rho_{st}}(F_{\rho_{00}}(x)) dF_{\rho_{00}}(x)\\
=&\int \partial_s F^{-1}_{\rho_{st}}(z) \cdot \partial_t F^{-1}_{\rho_{st}}(z) dz. 
\end{split}
\end{equation*}
From the above calculations, we recast the variational problem \eqref{mws} into the minimization problem \eqref{lmws1}, in terms of inverse cumulative distributions.  
\end{proof}

We next derive the optimality condition for variational problem \eqref{lmws1}. To simplify the notation, we denote 
\begin{equation*}
Z(s,t,x)=F^{-1}_{\rho_{st}}(x),
\end{equation*}
and write
\begin{equation*}
\begin{split}{C}(s,t,Z)
=\sqrt{\int \|\partial_s Z(s,t,x)\|^2dx\cdot \int\|\partial_t Z(s,t,x)\|^2 dx-(\int \partial_sZ(s,t,x)\cdot \partial_t Z(s,t,x) dx)^2 }.  
\end{split}
\end{equation*}
\begin{proposition}
The Euler-Lagrange equation of variational problem \eqref{lmws1} can be expressed in the form of the following second-order PDE, 
\begin{equation*}
\begin{split}
& \partial_{s}\Big(\partial_sZ \cdot {C}(s,t,Z)^{-1}\int \|\partial_tZ\|^2dx\Big)+ \partial_{t}\Big(\partial_tZ \cdot {C}(s,t,Z)^{-1} \int \|\partial_sZ\|^2dx\Big)\\
-&\partial_t\Big( \partial_{s}Z\cdot {C}(s,t,Z)^{-1}\int \partial_s Z\cdot\partial_tZdx \Big)-\partial_s\Big( \partial_{t}Z\cdot {C}(s,t,Z)^{-1}\int \partial_s Z\cdot\partial_tZdx \Big) =0,
\end{split}
\end{equation*}
with fixed boundary conditions 
\begin{equation*}  
Z(0,t,\cdot), \quad Z(1, t, \cdot), \quad Z(s,0, \cdot), \quad Z(s, 1, \cdot), \quad s,t\in [0,1].
\end{equation*}
\end{proposition}
\begin{proof}
We derive the Euler-Lagrange equation by computing 
\begin{equation*}
\frac{d}{d\epsilon}\int_0^1\int_0^1{C}(s,t,Z+\epsilon \delta Z)dsdt|_{\epsilon=0}=0,
\end{equation*}
where $\epsilon \in\mathbb{R}_+$ and $\delta Z\in C^{\infty}([0,1]^2, \mathbb{R})$ is a smooth testing function with $\delta Z(0, t, x)=\delta Z(1, t,x)=\delta Z(s, 0,x)=\delta Z(s,1,x)=0$ for $s,t\in [0,1]$. 
Note that 
\begin{equation*}
\begin{split}
&\frac{d}{d\epsilon}\int_0^1\int_0^1{C}(s,t, Z+\epsilon \delta Z)dsdt|_{\epsilon=0}\\
=&\frac{d}{d\epsilon}\int_0^1\int_0^1\Big(\int \|\partial_s Z+\epsilon \partial_s\delta Z\|^2dx\cdot \int\|\partial_t Z+\epsilon\partial_t\delta Z\|^2 dx\\
&\hspace{2cm}-(\int [\partial_sZ+\epsilon\partial_s\delta Z]\cdot [\partial_t Z+\epsilon\partial_t\delta Z] dx)^2 \Big)^{\frac{1}{2}}ds dt|_{\epsilon=0}\\
=&\frac{1}{2}\int_0^1\int_0^1{C}(s,t,Z)^{-1}\Big(2\int \partial_sZ \cdot\partial_s\delta Z dx\int \|\partial_tZ\|^2dx+2\int \partial_tZ \cdot\partial_t\delta Z dx\int \|\partial_sZ\|^2dx\\
&\hspace{4cm}-2(\int \partial_s Z\cdot\partial_tZdx)\cdot \int (\partial_s\delta Z \cdot \partial_tZ +\partial_t\delta Z \cdot \partial_sZ )dx \Big) dsdt. 
\end{split}
\end{equation*}
By applying integration by parts and choosing any smooth test function $\delta Z$, we derive the Euler-Lagrange equation. 
\end{proof}

\subsection{Wasserstein surfaces in higher dimensional spaces}
We next study the minimal Wasserstein surface problems in Lagrangian coordinates for general high-dimensional spaces. In this case, the variational problems (Problems \ref{prob1} and \ref{prob2}) no longer coincide. For the simplicity of exposition, we only work out the case of Problem \ref{prob1}. 

Consider $T\colon [0,1]^2\times\mathbb R^n\rightarrow\mathbb{R}^n$, such that $\rho(s,t,\cdot)= T(s,t,\cdot)_\# \rho(0,0,\cdot)$, equivalently,
\begin{equation*}
\rho(s,t, T(s,t,x)) \mathrm{det}(\nabla_xT(s,t,x))=\rho(0,0,x). 
\end{equation*}
We now reformulate the minimal surface problem \eqref{mws} or \eqref{mwss} in terms of pushforward mapping functions. 
\begin{proposition}[Minimal Wasserstein surface problems in Lagrangian coordinates]
The variational problem \eqref{mws} is equivalent to the following minimization problem:  
\begin{equation}\label{llmws1}
\begin{split}
&\inf_{T}\quad \int_0^1\int_0^1\Big(\int \|\partial_s T(s,t,x)\|^2\rho_{00}(x)dx\cdot \int\|\partial_t T(s,t,x)\|^2\rho_{00}(x) dx\\
&\hspace{3cm}-(\int \partial_s T(s,t,x)\cdot \partial_t T(s,t,x) \rho_{00}(x)dx)^2 \Big)^{\frac{1}{2}}ds dt,
\end{split}
\end{equation}
where the minimum is among all mapping functions $T(s,t,x)$, $s, t\in [0,1]$, with fixed boundary conditions: 
\begin{equation*}   
T(0,t,\cdot)_\#\rho_{00}=\rho_{0t}, \quad T(1, t, \cdot)_\#\rho_{00}=\rho_{1t}, \quad T(s,0, \cdot)_\#\rho_{00}=\rho_{s0}, \quad T(s, 1, \cdot)_\#\rho_{00}=\rho_{s1}, \quad s,t\in [0,1].
\end{equation*}
\end{proposition}
\begin{proof}
The proof follows from a similar argument in one dimension space. The continuity equation \eqref{mws2} can be reformulated as below. 
Denote
\begin{equation*}
\frac{\partial}{\partial s}T(s,t,x)=v_s(s,t, T(s,t,x)),\quad  \frac{\partial}{\partial t}T(s,t,x)=v_t(s,t, T(s,t,x)).
\end{equation*}
And the objective functional \eqref{mws1} can be reformulated in terms of mapping functions. In other words, 
\begin{equation*}
\begin{split}
\int \|v_s\|^2\rho dx=& \int \|v_s(s,t,y)\|^2\rho(s,t,y)dy\\
=& \int \|v_s(s,t,T(s,t,x))\|^2\rho(s,t,T(s,t,x))dT(s,t,x)\\
=& \int \|v_s(s,t,T(s,t,x))\|^2\rho(s,t,T(s,t,x))\textrm{det}(\nabla_xT(s,t,x))dx\\
=&\int \|\partial_sT(s,t,x)\|^2\rho_{00}(x)dx, 
\end{split}
\end{equation*}
where we let $y=T(s,t,x)$ in the first equality. 
Similarly, 
\begin{equation*}
\int \|v_t\|^2\rho dx=\int \|\partial_tT(s,t,x)\|^2\rho_{00}(x)dx. 
\end{equation*}
And 
\begin{equation*}
\begin{split}
\int v_s\cdot v_t\rho dx=& \int  v_s(s,t,y)\cdot v_t(s,t,y)\rho(s,t,y)dy\\
=& \int v_s(s,t,T(s,t,x))\cdot v_t(s,t, T(s,t,x))\rho(s,t,T(s,t,x))dT(s,t,x)\\
=& \int v_s(s,t,T(s,t,x))\cdot v_t(s,t, T(s,t,x))\rho(s,t,T(s,t,x))\mathrm{det}(\nabla_xT(s,t,x))dx\\
=&\int \partial_sT(s,t,x)\cdot \partial_tT(s,t,x)\rho_{00}(x)dx. 
\end{split}
\end{equation*}
The above calculations show thatn \eqref{mws} can be expressed as in \eqref{llmws1}.\end{proof}

We next derive the optimality condition for variational problem \eqref{llmws1}. Denote 
\begin{equation*}
\begin{split}{D}(s,t,T)
=&\Big\{\int \|\partial_s T(s,t,x)\|^2\rho_{00}(x)dx\cdot \int\|\partial_t T(s,t,x)\|^2\rho_{00}(x) dx\\
&-(\int \partial_s T(s,t,x)\cdot \partial_t T(s,t,x) \rho_{00}(x)dx)^2 \Big\}^{\frac{1}{2}}.  
\end{split}
\end{equation*}
\begin{proposition}
The Euler-Lagrange equation of variational problem \eqref{llmws1} satisfies the following second-order PDE: 
\begin{equation*}
\begin{split}
&\partial_{s}\Big(\partial_sT \cdot {D}(s,t,Z)^{-1}\int \|\partial_tT\|^2\rho_{00}dx\Big)+ \partial_{t}\Big(\partial_tT \cdot {D}(s,t,Z)^{-1} \int \|\partial_sT\|^2\rho_{00}dx\Big)\\
-&\partial_t\Big( \partial_{s}T\cdot {D}(s,t,Z)^{-1}\int \partial_s T\cdot\partial_tT\rho_{00}dx \Big)-\partial_s\Big( \partial_{t}T\cdot {D}(s,t,Z)^{-1}\int \partial_s T\cdot\partial_tT\rho_{00}dx \Big) =0,
\end{split}
\end{equation*}
with fixed boundary conditions: 
\begin{equation*}  
T(0,t,\cdot)_\#\rho_{00}=\rho_{0t}, \quad T(1, t, \cdot)_\#\rho_{00}=\rho_{1t}, \quad T(s,0, \cdot)_\#\rho_{00}=\rho_{s0}, \quad T(s, 1, \cdot)_\#\rho_{00}=\rho_{s1}, \quad s,t\in [0,1].
\end{equation*}
\end{proposition}
\begin{proof}
We derive the Euler-Lagrange equation by computing 
\begin{equation*}
\frac{d}{d\epsilon}\int_0^1\int_0^1{D}(s,t,T+\epsilon \delta T)dsdt|_{\epsilon=0}=0,
\end{equation*}
where $\epsilon \in\mathbb{R}_+$ and $\delta T\in C^{\infty}([0,1]^2, \mathbb{R}^n)$ is a smooth testing function with $\delta T(0, t, x)=\delta T(1, t,x)=\delta T(s, 0,x)=\delta T(s,1,x)=0$ for $s,t\in [0,1]$. 
Note that 
\begin{equation*}
\begin{split}
&\frac{d}{d\epsilon}\int_0^1\int_0^1{D}(s,t, T+\epsilon \delta T)dsdt|_{\epsilon=0}\\
=&\frac{d}{d\epsilon}\int_0^1\int_0^1\Big(\int \|\partial_s T+\epsilon \partial_s\delta T\|^2\rho_{00}dx\cdot \int\|\partial_t T+\epsilon\partial_t\delta T\|^2 \rho_{00}dx\\
&\hspace{3cm}-(\int [\partial_sT+\epsilon\partial_s\delta T]\cdot [\partial_t T+\epsilon\partial_t\delta T] \rho_{00}dx)^2 \Big)^{\frac{1}{2}}ds dt|_{\epsilon=0}\\
=&\frac{1}{2}\int_0^1\int_0^1{D}(s,t,Z)^{-1}\Big(\quad 2\int \partial_s T \cdot\partial_s\delta T\rho_{00} dx\int \|\partial_tT\|^2\rho_{00}dx\\
&\hspace{4cm}+2\int \partial_tT \cdot\partial_t\delta T\rho_{00}dx\int \|\partial_sT\|^2\rho_{00}dx\\
&\hspace{4cm}-2(\int \partial_s T\cdot\partial_tT\rho_{00}dx)\cdot \int (\partial_s\delta T \cdot \partial_tT +\partial_t\delta T\cdot \partial_sT )\rho_{00}dx \Big) dsdt. 
\end{split}
\end{equation*}
We derive the Euler-Lagrange equation by choosing any smooth test function $\delta Z$ and applying integration by parts formulas. 
\end{proof}

\section{Examples in Gaussian distributions}\label{sec4}
In this section, we present special examples of the minimal surface problem \eqref{mws} or \eqref{mwss}. We choose the boundary set of the Wasserstein rectangle to consist of Gaussian distributions with zero means and positive definite covariances. In this case, the minimal surface remains Gaussian, as a two parameter family. We analytically solve several special examples.

We first formulate the finite-dimensional counterpart of \eqref{mws} in terms of covariance matrices. 
\begin{proposition}{Under the above conditions on having a boundary of a Wasserstein rectangle consisting of zero-mean Gaussian distributions, problem \eqref{mws} reduces to the problem of determining
a two-parameter family of covariance matrices $\Sigma\in C^{2}([0,1]^2;\mathbb{R}^{n\times n})$ that solves the following minimization problem:}
\begin{subequations}\label{Gms}
\begin{equation}\label{Gms1}
\inf_{A_s, A_t, \Sigma}\quad \int_0^1\int_0^1\sqrt{\mathrm{tr}(\Sigma A_s^{\ts}A_s)\cdot \mathrm{tr}(\Sigma A_t^{\ts} A_t)-[\frac{1}{2}\mathrm{tr}(\Sigma(A_s^{\ts}A_t+A_t^{\ts}A_s))]^2}ds dt,
\end{equation}
where the minimization is over all $2$-parameter covariances $\Sigma=\Sigma(s,t)$, $(s,t)\in [0,1]^2$, and matrices $A_s=A_s(s,t)$, $A_t=A_t(s,t)\in\mathbb{R}^{n\times n}$, such that continuity equations (expressed in terms of covariance matrices as continuous-time Lyapunov equations),
\begin{equation}\label{Gms2}
\partial_s\Sigma=\Sigma A_s^{\ts}+A_s\Sigma,\quad \partial_t\Sigma=\Sigma A_t^{\ts}+A_t\Sigma, 
\end{equation}
hold with the specified boundary conditions
\begin{equation*}   
\Sigma(0,t), \quad \Sigma(1,t), \quad \Sigma(s,0), \quad \Sigma(s, 1), \mbox{ positive definite  for }s, t\in [0,1].
\end{equation*}
\end{subequations}
\end{proposition}
\begin{remark}
{The variational problem \eqref{mwss} corresponds to further requiring that
$A_s$ and $A_t$ are symmetric.}
\end{remark}
\begin{proof}
The proof follows a direct verification. Let 
\begin{equation*}
   \rho(s,t,x)=\frac{1}{(2\pi\cdot \mathrm{det}(\Sigma(s,t)))^{\frac{n}{2}}} e^{-\frac{1}{2}x^{\ts}\Sigma(s,t)^{-1}x},  
\end{equation*}
and denote 
\begin{equation*}
   v_s(s,t,x)=A_s(s,t)x,\quad v_t(s,t,x)=A_t(s,t)x.  
\end{equation*}
One can check that variational problem \eqref{Gms} is equivalent to variational problem \eqref{mws}. We first check the constraint, which is the continuity equation:
\begin{equation*}
\begin{split}
 0=&\frac{1}{\rho}(\partial_s\rho+\nabla\cdot(\rho v_s) )\\
=&\partial_s\log\rho+(\nabla\log\rho, v_s)+\nabla\cdot v_s\\
=&-\frac{n}{2}\partial_s \log\mathrm{det}(\Sigma)-\frac{1}{2}x^{\ts}\partial_s\Sigma^{-1}x-(\Sigma^{-1}x, A_sx)+\nabla\cdot(A_sx),
\end{split}
\end{equation*}
where we use the fact that $\frac{\partial_s\rho}{\rho}=\partial_s\log\rho$ and $\frac{\nabla \rho}{\rho}=\nabla\log\rho$. This implies that 
\begin{equation*}
\partial_s\Sigma^{-1}+\Sigma^{-1}A_s+A_s^{\ts}\Sigma^{-1}=0.   \end{equation*}
From the fact that $\partial_s\Sigma^{-1}=-\Sigma^{-1}\cdot\partial_s\Sigma\cdot\Sigma^{-1}$, we have 
\begin{equation*}
 \partial_s\Sigma=A_s\Sigma +\Sigma A_s^{\ts}.    
\end{equation*}
Similarly, we can obtain $\partial_t\Sigma=A_t\Sigma +\Sigma A_t^{\ts}$. We next rewrite the objective functional \eqref{mws1}. Note that 
\begin{equation*}
\begin{split}
    \int \|v_t\|^2\rho dx=&\int (A_tx, A_tx) \frac{1}{(2\pi\cdot \mathrm{det}(\Sigma(s,t)))^{\frac{n}{2}}} e^{-\frac{1}{2}x^{\ts}\Sigma(s,t)^{-1}x} dx\\ 
    =& \mathrm{tr}(\Sigma A_t^{\ts}A_t). 
\end{split}
\end{equation*}
Similarly, 
\begin{equation*}
    \begin{split}
    \int \|v_s\|^2\rho dx=&\int (A_sx, A_sx) \frac{1}{(2\pi\cdot \mathrm{det}(\Sigma(s,t)))^{\frac{n}{2}}} e^{-\frac{1}{2}x^{\ts}\Sigma(s,t)^{-1}x} dx\\ 
    =& \mathrm{tr}(\Sigma A_s^{\ts}A_s),
\end{split}
\end{equation*}
and 
\begin{equation*}
    \begin{split}
    \int v_t\cdot v_s\rho dx=&\int (A_sx, A_tx) \frac{1}{(2\pi\cdot \mathrm{det}(\Sigma(s,t)))^{\frac{n}{2}}} e^{-\frac{1}{2}x^{\ts}\Sigma(s,t)^{-1}x} dx\\ 
    =&\int \frac{1}{2}x^{\ts}(A_s^{\ts}A_t+A_t^{\ts}A_s)x\cdot \frac{1}{(2\pi\cdot \mathrm{det}(\Sigma(s,t)))^{\frac{n}{2}}} e^{-\frac{1}{2}x^{\ts}\Sigma(s,t)^{-1}x} dx\\ 
    =& \frac{1}{2}\mathrm{tr}(\Sigma (A_s^{\ts}A_t+A_t^{\ts}A_s)). 
\end{split}
\end{equation*}
\end{proof}
We next derive the critical-point system for  \eqref{Gms}. Denote 
\begin{equation*}
J=J(\Sigma, A_s,A_t)=\sqrt{\mathrm{tr}(\Sigma A_s^{\ts}A_s)\cdot \mathrm{tr}(\Sigma A_t^{\ts} A_t)-[\frac{1}{2}\mathrm{tr}(\Sigma(A_s^{\ts}A_t+A_t^{\ts}A_s))]^2}.   
\end{equation*}
\begin{proposition}
Denote $S_s$, $S_t\colon [0, 1]^2\rightarrow\mathbb{R}^{n\times n}$ as Lagrangian multipliers of equation \eqref{Gms2} for $\partial_s\Sigma$, $\partial_t\Sigma$, respectively. Then $S_s$, $S_t$ are symmetric matrices, i.e., $S_s=S_s^{\ts}$, $S_t=S_t^{\ts}$, and the  critical-point system of variational problem \eqref{Gms} consists of 
\begin{subequations}\label{mW}
\begin{equation}\label{mW1}
\left\{\begin{aligned}
&J^{-1}\Big(A_s \cdot \mathrm{tr}(\Sigma A_t^{\ts}A_t)-\frac{1}{2}A_t \cdot  \mathrm{tr}(\Sigma (A_s^{\ts}A_t+A_t^{\ts}A_s))\Big)=2S_s,\\ 
&J^{-1}\Big(A_t \cdot \mathrm{tr}(\Sigma A_s^{\ts}A_s)-\frac{1}{2}A_s \cdot  \mathrm{tr}(\Sigma (A_s^{\ts}A_t+A_t^{\ts}A_s))\Big)=2S_t, 
\end{aligned}\right.
\end{equation}
and 
\begin{equation}\label{mW2}
\left\{\begin{aligned}
&\partial_s S_s+\partial_t S_t
+\frac{1}{2}J^{-1}\Big\{A_s^{\ts}A_s\cdot \mathrm{tr}(\Sigma A_t^{\ts}A_t)+A_t^{\ts}A_t\cdot\mathrm{tr}(\Sigma A_s^{\ts}A_s)\\
&\hspace{3.5cm}-\frac{1}{2}(A_s^{\ts}A_t+A_t^{\ts}A_s)\cdot\mathrm{tr}(\Sigma (A_s^{\ts}A_t+A_t^{\ts}A_s))\Big\}=0,\\
&\partial_s\Sigma=\Sigma A_s^{\ts}+A_s\Sigma,\\
&\partial_t\Sigma=\Sigma A_t^{\ts}+A_t\Sigma. 
\end{aligned}\right. 
\end{equation}
\end{subequations}
\end{proposition}
\begin{proof}
Note that $\partial_s\Sigma$, $\partial_t\Sigma$ in equation \eqref{Gms2} are symmetric matricial equations. The corresponding Lagrangian multipliers can be chosen as symmetric matrices $S_s$, $S_t\in\mathbb{R}^{n\times n}$. I.e., $S_s=S_s^{\ts}$, $S_t=S_t^{\ts}$. We formulate the saddle point problem of variational problem \eqref{Gms},
\begin{equation*}
\inf_{\Sigma, A_s, A_t}\sup_{S_s, S_t}\quad \mathcal{L}_2(\Sigma, A_s,A_t,S_s, S_t)
\end{equation*}
where we denote $\mathcal{L}_2=\mathcal{L}_2(\Sigma, A_s,A_t,S_s, S_t)$ with 
\begin{equation*}
  \mathcal{L}_2=\int_0^1\int_0^1 J+ \mathrm{tr}\Big(S_s (\partial_s\Sigma-\Sigma A_s^{\ts}-A_s\Sigma)\Big) + \mathrm{tr}\Big(S_t (\partial_t\Sigma-\Sigma A_t^{\ts}-A_t\Sigma)\Big)  ds dt.   
\end{equation*}
By computing the saddle point of $\mathcal{L}_2$, we obtain 
\begin{equation}\label{saddle}
\left\{\begin{aligned}
&\frac{\delta}{\delta \Sigma}\mathcal{L}_2=0,\\
&\frac{\delta}{\delta A_s}\mathcal{L}_2=0,\\
&\frac{\delta}{\delta A_t}\mathcal{L}_2=0,\\
&\frac{\delta}{\delta S_s}\mathcal{L}_2=0,\\
&\frac{\delta}{\delta S_t}\mathcal{L}_2=0,
\end{aligned}\right. \Rightarrow   
\left\{\begin{aligned}
&\frac{1}{2}J^{-1}\Big(A_s^{\ts}A_s\mathrm{tr}(\Sigma A_t^{\ts}A_t)+A_t^{\ts}A_t\mathrm{tr}(\Sigma A_s^{\ts}A_s)\\
&\hspace{1cm}-\frac{1}{2}(A_s^{\ts}A_t+A_t^{\ts}A_s)\mathrm{tr}(\Sigma (A_s^{\ts}A_t+A_t^{\ts}A_s))\Big)\\
&-\partial_s S_s-\partial_t S_t-2A_s^{\ts}S_s-2A_t^{\ts}S_t=0,\\
&\frac{1}{2}J^{-1}\Big(2A_s\Sigma \cdot \mathrm{tr}(\Sigma A_t^{\ts}A_t)-A_t\Sigma\cdot  \mathrm{tr}(\Sigma (A_s^{\ts}A_t+A_t^{\ts}A_s))\Big)-2S_s\Sigma=0,\\
&\frac{1}{2}J^{-1}\Big(2A_t\Sigma \cdot \mathrm{tr}(\Sigma A_s^{\ts}A_s)-A_s\Sigma\cdot  \mathrm{tr}(\Sigma (A_s^{\ts}A_t+A_t^{\ts}A_s))\Big)-2S_t\Sigma=0,\\
&\partial_s\Sigma-\Sigma A_s^{\ts}-A_s\Sigma=0,\\
&\partial_t\Sigma-\Sigma A_t^{\ts}-A_t\Sigma=0. 
\end{aligned}\right. 
\end{equation}
For the third equality, we have 
\begin{equation*}
J^{-1}\Big(A_s \cdot \mathrm{tr}(\Sigma A_t^{\ts}A_t)-\frac{1}{2}A_t \cdot  \mathrm{tr}(\Sigma (A_s^{\ts}A_t+A_t^{\ts}A_s))\Big)-2S_s=0. 
\end{equation*}
Similarly, for the fourth equality, we have 
\begin{equation*}
J^{-1}\Big(A_t \cdot \mathrm{tr}(\Sigma A_s^{\ts}A_s)-\frac{1}{2}A_s \cdot  \mathrm{tr}(\Sigma (A_s^{\ts}A_t+A_t^{\ts}A_s))\Big)-2S_t=0. 
\end{equation*}
We finish the proof by applying the above equations to the saddle point system \eqref{saddle}. 
\end{proof}
We claim that the  critical-point system \eqref{mW} also solves the PDE system \eqref{csaddle} for the minimal surface problem \eqref{mws}. 
\begin{proposition}
 Suppose $(\Sigma(s,t), A_s(s,t), A_t(s,t), S_s(s,t), S_t(s,t))$ solves system \eqref{mW}. Denote 
\begin{equation*}
 \rho(s,t,x)=\frac{1}{(2\pi\cdot \mathrm{det}(\Sigma(s,t)))^{\frac{n}{2}}} e^{-\frac{1}{2}x^{\ts}\Sigma(s,t)^{-1}x}, \quad v_s(s,t,x)=A_s(s,t) x, \quad v_t(s,t,x)=A_t(s,t)x,    
\end{equation*}
and 
\begin{equation*}
 \Phi_s(s,t,x)=x^{\ts}S_s(s,t)x, \quad \Phi_t(s,t,x)=x^{\ts}S_t(s,t)x.    
\end{equation*}
Then $(\rho, v_s, v_t, \Phi_s, \Phi_t)$ solves the PDE system \eqref{csaddle}. 
\end{proposition}
The proof follows standard calculations and is omitted.

\subsection{Explicit examples}
We next derive explicit solutions in special cases for the variational problem \eqref{Gms}. Specifically, from now on, we assume that the rectangular boundaries are Gaussian distributions with diagonal covariance matrices. 
\begin{proposition}
Suppose $\Sigma(0,t)$, $\Sigma(1,t)$, $\Sigma(s,0)$, $\Sigma(s, 1)$ are diagonal positive definite matrices, for $s, t\in [0,1]$. 
Then variational problem \eqref{Gms} can be recast as
\begin{equation}\label{mwss1}
\inf_{\Sigma}\quad \int_0^1\int_0^1\sqrt{\mathrm{tr}((\partial_t\sqrt{\Sigma})^2)\cdot\mathrm{tr}((\partial_s\sqrt{\Sigma})^2)-[\mathrm{tr}(\partial_s\sqrt{\Sigma}\cdot\partial_t\sqrt{\Sigma})]^2} ds dt,
\end{equation}
where the minimization is over diagonal covariances $\Sigma=\Sigma(s,t)=\textrm{diag}(\Sigma_{ii}(s,t))_{1\leq i\leq n}$, $(s,t)\in [0,1]^2$ with fixed boundaries of positive definite covarianceas.    
\end{proposition}
\begin{proof}
Assume that $\Sigma(s,t)$ is a diagonal matrix for $s,t\in [0,1]$. From \eqref{Gms2} we have 
\begin{equation*}
\partial_s\Sigma_{ii}=2\Sigma_{ii}A_{s, ii}, 
\end{equation*}
and thus,
\begin{equation} \partial_s\sqrt{\Sigma_{ii}}=\frac{1}{2}(\Sigma_{ii})^{-\frac{1}{2}}\partial_s\Sigma_{ii}=(\Sigma_{ii})^{\frac{1}{2}}A_{s,ii}.   
\end{equation}
Therefore,
\begin{equation*}
\begin{split}
  \mathrm{tr}(\Sigma A_s^{\ts}A_s)=&\sum_{i=1}^n A_{s,ii}^2\Sigma_{ii}=\sum_{i=1}^n(\partial_s\sqrt{\Sigma_{ii}})^2=\|\partial_s\gamma\|^2,\\
  \mathrm{tr}(\Sigma A_t^{\ts}A_t)=&\sum_{i=1}^n(\partial_t\sqrt{\Sigma_{ii}})^2=\|\partial_t\gamma\|^2,  
  \end{split}
\end{equation*}
and 
\begin{equation*}
    \frac{1}{2}\mathrm{tr}(\Sigma (A_t^{\ts}A_s+A_s^{\ts}A_t))=\sum_{i=1}^n A_{s, ii}A_{t,ii}\Sigma_{ii}=\sum_{i=1}^n\partial_s\sqrt{\Sigma_{ii}}\cdot \partial_t\sqrt{\Sigma_{ii}}=\partial_s\gamma\cdot \partial_t\gamma. 
\end{equation*}
{Then, \eqref{Gms} reduces to \eqref{mwss1}, when $\Sigma(s,t)$ are diagonal matrices. Moreover, one can check that if $\Sigma$, $A_s$, $A_t$, $S_s$, $S_t$ are diagonal, it also solves the matrix system \eqref{mW}, which also satisfies the critical system of \eqref{mwss1}. This finishes the proof.}  
\end{proof}

Lastly, we present several closed-form solutions of minimal Wassertein surface problems in terms of Gaussian distributions with diagonal covariance matrices. For the simplicity of discussion we let $n=3$. Also, let $\gamma=\gamma(s,t)=(\gamma_i(s,t))_{i=1}^n\in\mathbb{R}^3$ with $\gamma_i(s,t)=\sqrt{\Sigma_{ii}(s,t)}$, $i=1,2,3$, and rewrite \eqref{mwss1} in terms of $\gamma$ as
\begin{equation*}
\inf_{\gamma}\quad \int_0^1\int_0^1\sqrt{\|\partial_t\gamma(s,t)\|^2\cdot\|\partial_s\gamma(s,t)\|^2-(\partial_t\gamma(s,t)\cdot \partial_s\gamma(s,t))^2} ds dt, 
\end{equation*}
where the minimization is among all paths  $\gamma\colon[0,1]^2\rightarrow\mathbb{R}^3$. This forms exactly a finite-dimensional minimal surface problem \eqref{E_min_surface}. Once again, let
\[
\gamma(s,t)=(s,t, z(s,t))^{\ts}=(\Sigma_{11}, \Sigma_{22}, \Sigma_{33}(\Sigma_{11}, \Sigma_{22}))^{\ts},
\]
where $\Sigma_{33}\colon \mathbb{R}^2_+\rightarrow\mathbb{R}_+$ is a two-variable function that depends on $\Sigma_{11}$ and $\Sigma_{22}$. The minimizer of \eqref{mwss1} satisfies the minimal surface equation \eqref{ms}. In other words, 
\begin{equation*}
(1+|\partial_tz|^2)\partial^2_{ss}z-2\partial_tz\partial_sz\partial^2_{st}z+ (1+|\partial_sz|^2)\partial^2_{tt}z=0.    
\end{equation*}
Several explicit solutions can be presented from classical studies of minimal surface problems. We present most derivations here for completeness (cf.~\cite{AT}). 
\begin{example}[Wasserstein plane on Gaussian distributions]
Let 
\begin{equation*}
   z(s,t)=a_1s+a_2 t+a_3, 
\end{equation*}
where $a_1$, $a_2$, $a_3\in \mathbb{R}$ are constants determined by the boundary conditions of the covariance matrices. Clearly,  $\partial_{ss}^2z=\partial^2_{st}z=\partial^2_{tt}z=0$, which satisfies equation \eqref{ms}. In other words, the plane in the Gaussian distribution satisfies 
\begin{equation*}
\sqrt{\Sigma_{33}}=a_1\sqrt{\Sigma_{11}}+a_2\sqrt{\Sigma_{22}}+a_3. 
\end{equation*}
\end{example}

\begin{example}[Wasserstein Scherk's surface on Gaussian distributions]
Let 
\begin{equation*}
    z(s,t)=g(s)+h(t),
\end{equation*}
where $g$, $h\colon \mathbb{R}\rightarrow\mathbb{R}$ are smooth functions. In this case, equation \eqref{ms} satisfies 
\begin{equation*}
    (1+\dot h^2(t))\ddot g(s)+(1+\dot g^2(s))\ddot h(t)=0.
\end{equation*}
This implies 
\begin{equation*}
    \frac{\ddot g(s)}{1+\dot g^2(s)}=-\frac{\ddot h(t)}{1+\dot h^2(t)}=-c, 
\end{equation*}
where $c\in\mathbb{R}$. If $c=0$, we obtain a trivial plane solution in which $z$ is linear in terms of $s$ and $t$. If $c\neq 0$, solving 
\begin{equation*}
 \frac{\ddot g(s)}{1+\dot g^2(s)} =\frac{d}{ds}\mathrm{arc}\tan (\dot g(s))=-c,
\end{equation*}
we obtain 
\begin{equation*}
   \dot g(s)=-\tan(cs-k_1), 
\end{equation*}
where $k_1\in\mathbb{R}$ is a constant. Thus $g(s)=\frac{1}{c}\log\cos(cs-k_1)+a_1$. Similarly, we can solve for $h(t)$, such that 
\begin{equation*}
    h(t)=-\frac{1}{c}\log\cos(ct-k_2)+a_2,
\end{equation*}
where $k_2$, $a_2\in\mathbb{R}$ are constants. Hence 
\begin{equation*}
z(s,t)=\frac{1}{c}\log\frac{\cos(cs-k_1)}{\cos(ct-k_2)}+(a_1+a_2). 
\end{equation*}
In other words, 
\begin{equation*}
\sqrt{\Sigma_{33}}=\frac{1}{c}\log\frac{\cos(c\sqrt{\Sigma_{11}}-k_1)}{\cos(c\sqrt{\Sigma_{22}}-k_2)}+(a_1+a_2).   
\end{equation*}
This solution can be viewed as a Scherk’s surface in terms of Gaussian covariance matrices. 
\end{example}
\begin{example}[Wasserstein Catenoid on Gaussian distributions]
Let 
\begin{equation*}
z(s,t)=f(r), \quad\textrm{where}\quad r=s^2+t^2. 
\end{equation*}
By some computations \cite{AT}, equation \eqref{ms} satisfies 
\begin{equation*}
 r\ddot f(r)+\dot f(r)^3+\dot f(r)=0.    
\end{equation*}
Thus 
\begin{equation*}
  \frac{\ddot f(r)}{\dot f(r)}-\frac{\dot f(r)\ddot f(r)}{1+\dot f(r)^2}=\frac{\ddot f(r)}{\dot f(r)+\dot f(r)^3}=-\frac{1}{r}.   
\end{equation*}
Assume $\dot f(r)>0$. By choosing $r_1>0$, such that $r\geq r_1$, 
\begin{equation*}
 \log \dot f(r)-\log\sqrt{1+(\dot f(r))^2}=\log r_1-\log r.    
\end{equation*}
This implies 
\begin{equation*}
 \dot f(r)=\frac{1}{\sqrt{\frac{r^2}{r_1^2}-1}}.    
\end{equation*}
There are two solutions 
\begin{equation*}
  f(r)=c_1\pm r_1\mathrm{arc}\cosh(\frac{r}{r_1}),
\end{equation*}
where $c_1$, $r_1\in\mathbb{R}$ are constants. These are known as the Catenoid solution. In other words, 
\begin{equation*}
\sqrt{\Sigma_{33}}=c_1\pm r_1\mathrm{arc}\cosh(\frac{{\Sigma_{11}+\Sigma_{22}}}{r_1}). 
\end{equation*}
This solution can be viewed as a Catenoid in terms of Gaussian covariance matrices. 
\end{example}

\begin{example}[Wasserstein Helicoid on Gaussian distributions]
Let 
\begin{equation*}
z(s,t)=k(\xi), \quad \textrm{where}\quad \xi=\frac{t}{s}.    
\end{equation*}
Again, by some calculations in \cite{AT}, we have 
\begin{equation*}
 (1+\xi^2) \ddot k(\xi)+2\xi \dot k(\xi)=0.   
\end{equation*}
By separating variables, 
\begin{equation*}
\frac{\ddot k(\xi)}{\dot k(\xi)}=-\frac{2\xi}{1+\xi^2}.    
\end{equation*}
This implies the following first-order integral 
\begin{equation*}
 \log\dot k(\xi)=\log c_1-\log (1+\xi^2),    
\end{equation*}
where $c_1\in\mathbb{R}$ is a constant. This leads to a solution named the Helicoid. 
\begin{equation*}
 k(\xi)=c_1\mathrm{arc}\tan(\xi)+c_2,
\end{equation*}
where $c_1$, $c_2\in\mathbb{R}$ are constants. In other words, 
\begin{equation*}
 \sqrt{\Sigma_{33}}=c_1\mathrm{arc}\tan(\sqrt{\frac{\Sigma_{22}}{\Sigma_{11}}})+c_2. 
\end{equation*}
This solution can be viewed as a Helicoid in terms of Gaussian covariance matrices. 
\end{example}

\section{Concluding remarks}
The main point of the present work is to introduce the notion of minimal surface in the Wasserstein space of probability distributions $(\mathcal P_2,W_2)$. Besides a purely mathematical interest, motivation for considering surfaces stems from problems in thermodynamics where (finitely parametrized, so far) thermodynamic states traverse closed paths in $(\mathcal P_2,W_2)$ that enclose regions, where surface integrals capture available work over a cycle. In these, the perimeter quantifies dissipation, leading to isoperimetric problems \cite{miangolarra2021energy}.
Thus, our interest has been in advancing the notion of minimal Wasserstein surfaces with an eye toward carrying out such analyses in thermodynamics and possibly other disciplines in a parametric-free fashion.

Specifically, in the present paper, we introduce minimal surface problems in Wasserstein spaces to serve as a canonical choice of a surface enclosed by a Wasserstein curve. In the above development, we focused on minimal surfaces with borders geodesic curves, but more general boundaries are also of interest.

Our basic formulation calls for a two-parameter Benamou-Brenier-type variational problem in the space of probability densities. We derive the minimal surface equations and critical-point systems of the proposed variational problems. These are systems of two-parameter PDEs. Finite-dimensional examples of, e.g., Gaussian distributions, are provided. We derive several explicit solutions for Wasserstein minimal surfaces, such as planes, Scherk’s surfaces, Catenoids, and Helicoids, in terms of suitably parametrized covariance matrices of the respective Gaussian distributions. 

Several natural questions arise that are most open at present. In particular, the existence of solutions in the general case for the two-parameter variational problem of minimal Wasserstein surfaces is broadly open. Computation is natural next, as efficient schemes in high dimensions need to be developed. 
Wasserstein minimal surfaces, very much like any two-dimensional minimal surface, must admit a formulation based on vanishing mean curvature. Finally, in the more general setting where the boundary is a closed curve in Wasserstein space, what properties of the boundary can ensure the existence of a minimal surface?
\black

\bibliography{refs}
\bibliographystyle{siam}
\end{document}